\documentclass[11pt]{article}

\usepackage{times}
\usepackage{graphicx}
\usepackage{psfrag}
\usepackage{amsmath,amsthm}
\usepackage{amsfonts}

\newtheorem*{lema}{Lemma}
\newtheorem*{thma}{Theorem}

\theoremstyle{definition}

\theoremstyle{remark}

\begin{document}

\parskip=10pt

\flushbottom 

\title{Laminating lattices with symmetrical glue} 

\author{Veit Elser and Simon Gravel\\ Laboratory of Atomic and Solid State Physics\\
Cornell University, Ithaca, NY 14853-2501} 

\date{}

\maketitle

\begin{abstract}
We use the automorphism group $Aut(H)$, of holes in the lattice $L_8=A_2\oplus A_2\oplus D_4$, as the starting point in the construction of sphere packings in 10 and 12 dimensions. A second lattice,  $L_4=A_2\oplus A_2$, enters the construction because a subgroup of $Aut(L_4)$ is isomorphic to $Aut(H)$. The lattices $L_8$ and $L_4$, when glued together through this relationship, provide an alternative construction of the laminated lattice in twelve dimensions with kissing number 648. More interestingly, the action of $Aut(H)$ on $L_4$ defines a pair of invariant planes through which dense, non-lattice packings in 10 dimensions can be constructed. The most symmetric of these is aperiodic with center density $1/32$. These constructions were prompted by an unexpected arrangement of 378 kissing spheres discovered by a search algorithm.
\end{abstract}

\subsection*{Introduction}

One of the simplest geometrical constraint problems with strong ties to the design of codes is the problem of \textit{kissing spheres} \cite{SPLAG}. The object is to pack equal spheres --- as many as possible --- so that each sphere touches a given sphere of the same size. The best lower bounds on the maximum number of kissing spheres, in all dimensions where this problem has been studied, are derived from integral lattices or error correcting codes. Exceptionally symmetric lattices, such as $E_8$ and the Leech lattice, account for the relatively few dimensions where the maximum kissing number has been established.
Conversely, searches for good solutions to the kissing spheres problem by an unbiased algorithm tests the scope of the known schemes for designing good codes.

\newpage

This paper comes in response to some unexpected results in the search for kissing spheres in 10 dimensions \cite{D&C}. The algorithm used in the search worked directly with the two kinds of distance constraints on the sphere centers, $x_i$ : 
\[
\qquad \|x_i\|=2, \qquad\qquad \|x_i-x_j\|\ge 2, \;i\ne j\;.
\]
The only assumption that restricted the search was inversion symmetry (the operation $x_i\to -x_i$ merely interchanges pairs of spheres). A rough characterization of a solution is given by the distribution of the cosines of the angles subtended by pairs of spheres, $4\cos{(\theta_{ij})}=x_i\cdot x_j$.
Because of inversion symmetry the cosine distribution is symmetric about zero. 

\begin{figure}[t]
\includegraphics[width=5.in]{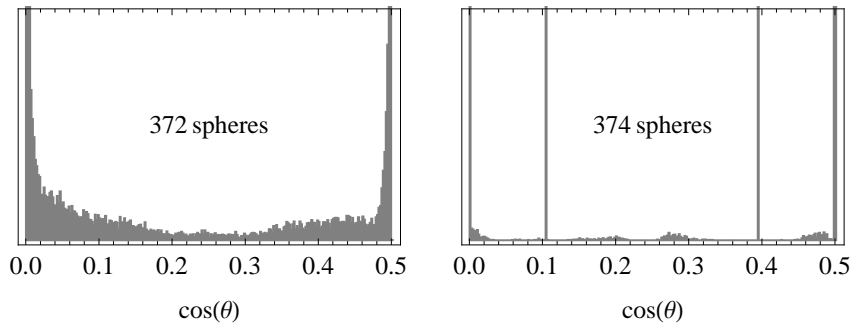}
\caption{Results of a numerical search for kissing spheres in 10 dimensions. Shown is the distribution of cosines of the angles subtended by pairs of spheres. The distribution for 374 spheres has a pair of sharp peaks at $\cos{(\theta)}=(3\pm\sqrt{3})/12$.}
\end{figure}

For small numbers of spheres there is considerable freedom of movement in the solution and the cosine distribution looks quasi-continuous. Figure 1 (left) shows a typical output produced by the algorithm for 372 spheres. The number 372 is of interest because this is the maximum kissing number of the densest known packing in ten dimensions, $P_{10c}$ \cite{SPLAG}. The densest known lattice packing has kissing number 336. Because there exists a non-lattice packing with kissing number 500 ($P_{10b}$), the quasi-continuous distribution found for 372 spheres is not in itself surprising. What came as a surprise was the distribution obtained for 374 spheres. There the appearance of sharp peaks, Fig. 1 (right), shows that the sphere arrangement has taken a very ordered form. Indeed, with the exception of a very small number of spheres, the cosines found by the algorithm are (within numerical precision) all contained in the set
$C=\{0,1/2, (3\pm\sqrt{3})/12\}$. The irrational values in this set are at odds with the simple rational values generated by integral lattices and error correcting codes.

A forensic examination of the sphere centers obtained by the algorithm has revealed that there is a maximal set of 378 spheres consistent with the cosines $C$. This was accomplished by first finding a suitable coordinate system that reproduced the cosines $C$ when the numerical coordinates were approximated by numbers of the form $a+b\sqrt{3}$, $a,b\in\mathbb{Q}$. A basis reduction algorithm was then used to establish that the 378 sphere centers form a $\mathbb{Z}$-module of rank 12. 

The main outcome of the forensic examination, and the theme of this paper, is a motive for the arrangement of the 378 kissing spheres. In rough outline, we first observe that a dense 12 dimensional lattice $L_{12}$ can be constructed by gluing together a pair of not exceptionally dense lattices $L_8$ and $L_4$ in eight and four dimensions. This gluing is the key to the construction of a dense non-lattice packing $Q_{10}$ in ten dimensions. The reduction in dimension occurs entirely within $L_4$, which possesses a symmetry that decomposes the four dimensional space into a pair of invariant planes. Moreover, as this symmetry is non-crystallographic in two dimensions, these planes are totally irrational subspaces with respect to $L_4$. The packing $Q_{10}$ results from the projection of $L_{12}$ into the orthogonal complement of one of these irrational planes and explains the rank-12 $\mathbb{Z}$-module discovered by the algorithm. What emerges as a reasonable motive is that the construction of $Q_{10}$, from the dense $L_{12}$, is such that it is able to preserve many of the sphere contacts of the lattice packing.

We will follow the standard terminology of lattices whenever possible \cite{SPLAG}. In particular, the \textit{norm} of a vector $x$ is its Euclidean inner product $x\cdot x$. The \textit{automorphism group} of a lattice, $Aut(\Lambda)$, is the group of unimodular transformations of the lattice generators that preserves the Euclidean inner product. Equivalently, $Aut(\Lambda)$ is the group of real orthogonal transformations on the lattice vectors that permutes the elements of $\Lambda$. We use the term \textit{glue group} for any subgroup of $\Lambda^\ast/\Lambda$,  where $\Lambda^\ast$ is the dual of $\Lambda$. The \textit{depth} of a general point $x$, in a lattice $\Lambda$, is the minimum norm in the set $\{x-y\colon y\in \Lambda, y\ne x\}$, and is denoted $\Delta(x)$. The number of minimum norm elements is written $\tau(x)$ and equals the kissing number of $\Lambda$ when $x\in\Lambda$.

\subsection*{Holes of the lattice $A_2\oplus A_2\oplus D_4$}

Glue groups  of particular interest are those generated by the vertices of the Voronoi region, the \textit{holes} of the lattice. Our construction begins with the lattice 
\[
L_8=A_2\oplus A_2\oplus D_4\;.
\]
$L_8$ has kissing number $6+6+24=36$ and for our choice of scale, minimum norm 4 and center density $\delta_8=(1/2\sqrt{3})^2(1/8)=1/96$. Each $A_2$ has two holes and the glue group $\mathbb{Z}_3$ has a single generator; $D_4$ has three holes and glue group $\mathbb{Z}_2\times \mathbb{Z}_2$ with two generators. The glue group $H$ generated by the holes of $L_8$ thus has four generators which, for later convenience, we represent as vectors under addition modulo 1:
\begin{equation}\label{glue_gen}
h_1=\begin{bmatrix}
\frac{1}{3}&\frac{1}{3}&\frac{1}{3}&\frac{1}{3}
\end{bmatrix}
\quad
h_2=\begin{bmatrix}
\frac{1}{3}&\frac{1}{3}&\frac{2}{3}&\frac{2}{3}
\end{bmatrix}
\end{equation}
\begin{equation*}
h_3=\begin{bmatrix}
\frac{1}{2}&0&\frac{1}{2}&0
\end{bmatrix}
\quad
h_4=\begin{bmatrix}
0&\frac{1}{2}&0&\frac{1}{2}
\end{bmatrix}
\end{equation*}
The generator $h_1$ represents a hole in one $A_2$ component while $h_2$ represents a hole in the other. All four elements, $\pm h_1$ and $\pm h_2$, thus have depth $4/3$ in $L_8$. Another four elements, $\pm h_1 \pm h_2$, represent holes in both $A_2$ components and have depth $8/3$. The vectors $h_3$, $h_4$, and $h_3+h_4$ represent the three holes in the $D_4$ component of $L_8$ with depth 2.

The 36 elements of $H$ fall into conjugacy classes with respect to $Aut(L_8)$. The quotient of $Aut(L_8)$, with respect to the normal subgroup that acts trivially on $H$, defines $Aut(H)$, the automorphism group of the holes. $Aut(H)$ is generated by (i) exchange of the two holes within either $A_2$ component, (ii) exchange of the two $A_2$ components, and (iii) any permutation of the three holes of $D_4$. The $2^2\times 2\times 3!$ elements of  $Aut(H)$ have an interesting representation in terms of $4\times 4$ matrices acting on the 4-component glue elements by right multiplication.
The matrices
\begin{equation*}
\sigma=
\begin{bmatrix}
1&0&0&0\\
0&1&0&0\\
0&0&-1&0\\
0&0&0&-1
\end{bmatrix}
\quad
\rho_4=
\begin{bmatrix}
0&0&-1&0\\
0&0&0&-1\\
1&0&0&0\\
0&1&0&0
\end{bmatrix}
\end{equation*}
together implement the exchange of $h_1$ and $h_2$, as well as sending these to their inverses (without affecting $h_3$ and $h_4$). Two other matrices generate the 6 permutations of the three holes of $D_4$:
\begin{equation*}
\rho_2=
\begin{bmatrix}
0&0&0&1\\
0&0&1&0\\
0&1&0&0\\
1&0&0&0
\end{bmatrix}
\quad
\rho_3=
\begin{bmatrix}
0&-1&0&0\\
1&-1&0&0\\
0&0&0&-1\\
0&0&1&-1
\end{bmatrix}.
\end{equation*}
Matrix $\rho_3$ cyclically permutes $h_3$, $h_4$ and $h_3+h_4$ (fixing $h_1$ and $h_2$),
while $\rho_2$ exchanges $h_3$ and $h_4$ (again with no effect on the $A_2$ holes when combined with an appropriate combination of $\sigma$ and $\rho_4$).

For future calculations we tabulate data on the 6 conjugacy classes of the $L_8$ glue elements in Table 1. As an example, the entry $h=\begin{bmatrix}
\frac{5}{6}&\frac{1}{3}&\frac{5}{6}&\frac{1}{3}
\end{bmatrix}=h_1+h_3$
is the combination of a hole in one $A_2$ component and a hole in $D_4$. There are $\tau_8(h)=3\times 8$ points in $L_8$ at the minimum squared distance $\Delta_8(h)=4/3 + 2$ from $h$. The symmetry orbit of $h$ has size 12 and is generated by $\rho_3 \rho_4$:
\begin{equation*}
\begin{bmatrix}
\frac{5}{6}&\frac{1}{3}&\frac{5}{6}&\frac{1}{3}
\end{bmatrix}
\quad
\begin{bmatrix}
\frac{1}{3}&\frac{5}{6}&\frac{2}{3}&\frac{1}{6}
\end{bmatrix}
\quad
\begin{bmatrix}
\frac{1}{6}&\frac{1}{6}&\frac{1}{6}&\frac{1}{6}
\end{bmatrix}
\quad
\begin{bmatrix}
\frac{1}{6}&\frac{2}{3}&\frac{5}{6}&\frac{1}{3}
\end{bmatrix}
\end{equation*}
\begin{equation*}
\begin{bmatrix}
\frac{1}{3}&\frac{5}{6}&\frac{1}{3}&\frac{5}{6}
\end{bmatrix}
\quad
\begin{bmatrix}
\frac{5}{6}&\frac{5}{6}&\frac{1}{6}&\frac{1}{6}
\end{bmatrix}
\quad
\begin{bmatrix}
\frac{1}{6}&\frac{2}{3}&\frac{1}{6}&\frac{2}{3}
\end{bmatrix}
\quad
\begin{bmatrix}
\frac{2}{3}&\frac{1}{6}&\frac{1}{3}&\frac{5}{6}
\end{bmatrix}
\end{equation*}
\begin{equation*}
\begin{bmatrix}
\frac{5}{6}&\frac{5}{6}&\frac{5}{6}&\frac{5}{6}
\end{bmatrix}
\quad
\begin{bmatrix}
\frac{5}{6}&\frac{1}{3}&\frac{1}{6}&\frac{2}{3}
\end{bmatrix}
\quad
\begin{bmatrix}
\frac{2}{3}&\frac{1}{6}&\frac{2}{3}&\frac{1}{6}
\end{bmatrix}
\quad
\begin{bmatrix}
\frac{1}{6}&\frac{1}{6}&\frac{5}{6}&\frac{5}{6}
\end{bmatrix}
\end{equation*}

\renewcommand\arraystretch{1.5}
\begin{table}
\begin{center}
\begin{tabular}[top]{cccc}
$h$ & orbit size & $\Delta_8(h)$ & $\tau_8(h)$\\
\hline
$\begin{matrix}
0&0&0&0
\end{matrix}$ & 1 & 4 & 36\\
$\begin{matrix}
\frac{1}{2}&0&\frac{1}{2}&0
\end{matrix}$ & 3 & 2 & 8\\
$\begin{matrix}
\frac{1}{3}&\frac{1}{3}&\frac{1}{3}&\frac{1}{3}
\end{matrix}$ & 4 & 4/3 & 3\\
$\begin{matrix}
\frac{1}{3}&\frac{1}{3}&0&0
\end{matrix}$ & 4 & 8/3 & 9\\
$\begin{matrix}
\frac{5}{6}&\frac{1}{3}&\frac{5}{6}&\frac{1}{3}
\end{matrix}$ & 12 & 10/3 & 24\\
$\begin{matrix}
\frac{5}{6}&\frac{1}{3}&\frac{1}{2}&0
\end{matrix}$ & 12 & 14/3 & 72
\end{tabular}
\caption{Properties of the symmetry orbits of glue in $L_8$. The orbit representatives $h$ are sums (mod 1) of the generators \eqref{glue_gen}; $\Delta_8(h)$ is the depth and $\tau_8(h)$ the number of points in $L_8$ at squared distance $\Delta_8(h)$.}
\end{center}
\end{table}

\subsection*{Symmetric glue for the lattice $A_2\oplus A_2$}

We now interpret the 36 glue elements, when expressed as 4-component vectors, as glue elements of a four dimensional lattice $L_4$. Since we want the glue in $L_4$ to include all the automorphisms the corresponding elements have in $L_8$, we construct $L_4$ so the $4\times 4$ matrices  $\rho_2$, $\rho_3$, $\rho_4$, and $\sigma$ are elements of $Aut(L_4)$. These matrices act on the four generators of $L_4$ by left multiplication. 

We exhibit $L_4$ in terms of the projections of its generators $u_1$, $u_2$, $u_3$, $u_4$ in two orthogonal planes. These are shown in Figure 2; components in the two planes are distinguished by superscripts ${\scriptstyle ||}$ and ${\scriptstyle \perp}$. The pair $u^{\scriptscriptstyle ||}_1, u^{\scriptscriptstyle ||}_2$ generates an $A_2$ lattice, as does the pair $u^{\scriptscriptstyle ||}_3, u^{\scriptscriptstyle ||}_4$. The same holds for the projection into the other plane. This much is consistent with the automorphism generated by $\rho_3$, which acts as a rotation by $120^\circ$ in the two planes. The pair $u^{\scriptscriptstyle ||}_3, u^{\scriptscriptstyle ||}_4$ is rotated by $+90^\circ$ relative to $u^{\scriptscriptstyle ||}_1, u^{\scriptscriptstyle ||}_2$, and by $-90^\circ$ in the other plane. If we fix the coordinates so that $u^{\scriptscriptstyle ||}_1=u^{\scriptscriptstyle \perp}_1$ and $u^{\scriptscriptstyle ||}_2=u^{\scriptscriptstyle \perp}_2$, then this implies $u^{\scriptscriptstyle ||}_3=-u^{\scriptscriptstyle \perp}_3$ and $u^{\scriptscriptstyle ||}_4=-u^{\scriptscriptstyle \perp}_4$. From this we have \begin{equation*}
u_i\cdot u_j=(u^{\scriptscriptstyle ||}_i+u^{\scriptscriptstyle \perp}_i)\cdot (u^{\scriptscriptstyle ||}_j+u^{\scriptscriptstyle \perp}_j)=0\qquad i\in\{1,2\},\; j\in\{3,4\}
\end{equation*}
thus identifying $L_4$ as $A_2\oplus A_2$. We choose the scale so that $L_4$ has minimal norm 4 and center density $\delta_4=1/12$. The actions of $\rho_2$ and $\rho_4$ on the projected lattice generators show that these also generate isometries of the two planes: $\rho_2$ reflects across a line and $\rho_4$ rotates by $90^\circ$. Altogether, $\rho_2$, $\rho_3$, $\rho_4$ generate the two-dimensional non-crystallographic reflection group $G_0$ of order 24. The final generator of $Aut(H)$, the involution $\sigma$, reverses the signs of $u_3$ and $u_4$ relative to $u_1$ and $u_2$, and is therefore conjugate to an exchange of the two planes. We observe that the only point of $L_4$ contained in either plane is the origin.

\begin{figure}[t]
\begin{center}
\includegraphics[width=5.in]{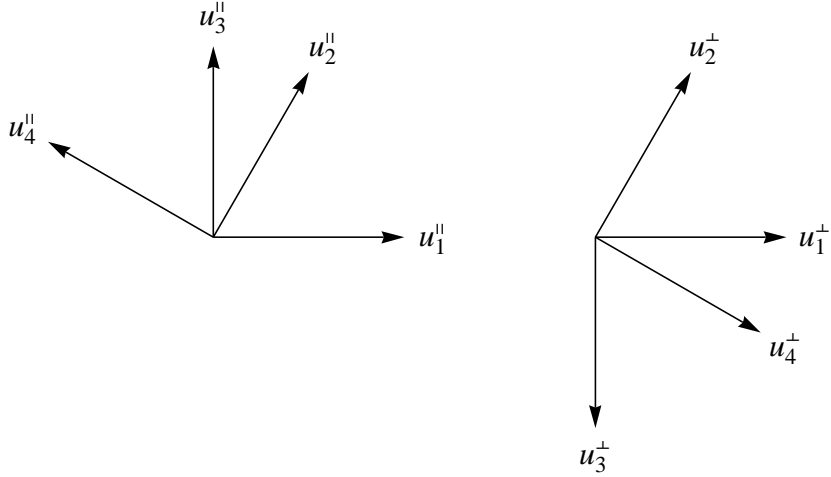}
\caption{Projections of the generators of $L_4$ into two orthogonal planes.}
\end{center}
\end{figure}

The data for the glue elements, as for $L_8$, again condenses into a table of the symmetry orbits. As a guide for the computations in Table 2, we show in Figure 3 the components of the glue in the $A_2$ plane spanned by $u_1$ and $u_2$ (the diagram for the other $A_2$ component is identical). There are 12 distinct glue elements in this component, and these fall into four families:
\begin{equation*}
\begin{array}{ccccccc}
A:&\left[0\;0\right]&&&&&\\
B:&\left[\frac{1}{2}\;0\right]&\left[0\;\frac{1}{2}\right]&\left[\frac{1}{2}\;\frac{1}{2}\right]&&&\\
C:&\left[\frac{1}{3}\;\frac{1}{3}\right]&\left[\frac{2}{3}\;\frac{2}{3}\right]&&&&\\
D:&\left[\frac{1}{6}\;\frac{1}{6}\right]&\left[\frac{1}{6}\;\frac{2}{3}\right]&\left[\frac{2}{3}\;\frac{1}{6}\right]&\left[\frac{5}{6}\;\frac{1}{3}\right]&\left[\frac{1}{3}\;\frac{5}{6}\right]&\left[\frac{5}{6}\;\frac{5}{6}\right]\\
\end{array}
\end{equation*}
Glue in the family $B$ has depth 1 from 2 points of $A_2$, family $C$ has depth $4/3$ from 3 points, and family $D$ has depth $1/3$ from only one point of $A_2$. Taking into account both $A_2$ components we obtain the information in Table 2.

\begin{figure}[t]
\begin{center}
\includegraphics[width=3.in]{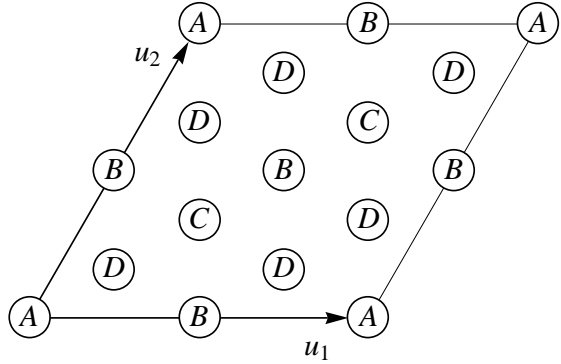}
\caption{The four types of glue in one $A_2$ component of the lattice $L_4$.
}
\end{center}
\end{figure}

\begin{table}
\begin{center}
\begin{tabular}[top]{cccc}
$h$ & orbit size & $\Delta_4(h)$ & $\tau_4(h)$\\
\hline
$\begin{matrix}
0&0&0&0
\end{matrix}$ & 1 & 4 & 12\\
$\begin{matrix}
\frac{1}{2}&0&\frac{1}{2}&0
\end{matrix}$ & 3 & 2 & 4\\
$\begin{matrix}
\frac{1}{3}&\frac{1}{3}&\frac{1}{3}&\frac{1}{3}
\end{matrix}$ & 4 & 8/3 & 9\\
$\begin{matrix}
\frac{1}{3}&\frac{1}{3}&0&0
\end{matrix}$ & 4 & 4/3 & 3\\
$\begin{matrix}
\frac{5}{6}&\frac{1}{3}&\frac{5}{6}&\frac{1}{3}
\end{matrix}$ & 12 & 2/3 & 1\\
$\begin{matrix}
\frac{5}{6}&\frac{1}{3}&\frac{1}{2}&0
\end{matrix}$ & 12 & 4/3 & 2
\end{tabular}
\caption{Properties of the symmetry orbits of glue in $L_4$. $\Delta_4(h)$ is the depth and $\tau_4(h)$ the number of points in $L_4$ at squared distance $\Delta_4(h)$.}
\end{center}
\end{table}

\subsection*{Twelve dimensional lattice packing}

For each element $h$ of the glue group $H$ we have a glue element $x_8(h)$ for the lattice $L_8$ and a corresponding glue element $x_4(h)$ for the lattice $L_4$. The elements $x_8(h)+x_4(h)$, $h\in H$, form a glue group for $L_8\oplus L_4$ and, together with $L_8\oplus L_4$, form a lattice $L_{12}$. From the data in Tables 1 and 2 we can check that $L_{12}$ has minimal norm 4 and compute the kissing number. 

Let $x$ and $y$ be distinct points of $L_{12}$. The vector $x-y$ belongs to one of the 36 conjugacy classes with respect to $L_8\oplus L_4$. A non-zero vector in the class $\left[0\;0\;0\;0\right]$ has minimal norm 4 because it must be a minimal vector of either $L_8$ or $L_4$; the number of such vectors is $36+12$, the sum of the kissing numbers of $L_8$ and $L_4$. If the conjugacy class $h$ of $x-y$ is non-trivial, then its minimal norm is $\Delta_8(h)+\Delta_4(h)$ and the number of such vectors is $\tau_8(h) \tau_4(h)$. Computations for all the conjugacy classes are given in Table 3 and verify that $L_{12}$ has minimal norm 4 and kissing number $48+96+108+108+288=648$. The center density of this lattice is $\delta_{12}=|H|\,\delta_8\, \delta_4=1/32$. These properties match those of the 12-dimensional laminated lattice having the highest kissing number, $\Lambda_{12}^\mathrm{max}$ \cite{SPLAG}. The latter is one of three, equal density lattices produced by laminating one dimension at a time, starting with $\Lambda_1=2\mathbb{Z}$ and always maximizing the density. $L_{12}$ and $\Lambda_{12}^\mathrm{max}$ were shown to be isomorphic by comparing their Gram matrices \cite{Sloane}.

\begin{table}
\begin{center}
\begin{tabular}[top]{ccccc}
$h$ & orbit size & $\Delta_8(h)+\Delta_4(h)$ & $\tau_8(h) \tau_4(h)$& number\\
\hline
$\begin{matrix}
\frac{1}{2}&0&\frac{1}{2}&0
\end{matrix}$ & 3 & 4 & 32 & 96\\
$\begin{matrix}
\frac{1}{3}&\frac{1}{3}&\frac{1}{3}&\frac{1}{3}
\end{matrix}$ & 4 & 4 & 27 & 108\\
$\begin{matrix}
\frac{1}{3}&\frac{1}{3}&0&0
\end{matrix}$ & 4 & 4 & 27 & 108\\
$\begin{matrix}
\frac{5}{6}&\frac{1}{3}&\frac{5}{6}&\frac{1}{3}
\end{matrix}$ & 12 & 4 & 24 & 288\\
$\begin{matrix}
\frac{5}{6}&\frac{1}{3}&\frac{1}{2}&0
\end{matrix}$ & 12 & 6 & 144 & 1728
\end{tabular}
\caption{The short vectors of $L_{12}$ in the non-trivial conjugacy classes of the glue group. }
\end{center}
\end{table}

\subsection*{Ten dimensional non-lattice packings}

Our non-standard construction of $\Lambda_{12}^\mathrm{max}\cong L_{12}$ leads rather directly to a non-lattice packing that explains the numerical kissing number results in 10 dimensions. The construction of the non-lattice  packing is closely related to quasicrystal patterns, of which the best known example is the Penrose tiling. The focus now shifts to the Euclidean space $X_4$ that contains $L_4$. The invariant planes of the the group $G_0$ provide a natural orthogonal decomposition $X_4=X_{||}\oplus X_\perp$. A general point $x\in L_{12}$ has a unique expression of the form $x=x_{||}+x_\perp+x_8$, where $x_8$ is in the space containing $L_8$. The only role of $x_8$ in the following construction is through its depth in $L_8$, which we denote by $\Delta_8(x)$; the depth depends only on the associated glue element.
The non-lattice packing is then obtained as follows:

\noindent\textbf{Construction Q}. Define the subset $S_{10}=\{x\in L_{12}\colon x_\perp\in V_\perp\}$, where $V_\perp\subset X_\perp$ is a bounded domain and specified in detail below. From $S_{10}$, a packing in 10 dimensions is given by the set $Q_{10}=\{\sqrt{2}\, x_{||}+x_8\colon x\in S_{10}\}$.

We will see that if there are kissing spheres at $x, y\in L_{12}$, that is, $z=x-y$ is a vector of norm 4, then very often $z_{||}\cdot z_{||} = z_\perp\cdot z_\perp$. Since this implies that $z_{||}+z_\perp+z_8$ and $\sqrt{2}\,z_{||}+z_8$ have the same norm, spheres that are kissing in $L_{12}$ will often have counterparts in $Q_{10}$ that are kissing as well. Note that if $z$ is in the conjugacy class  $\left[\frac{5}{6}\;\frac{1}{3}\;\frac{1}{2}\;0\right]$, then $(\sqrt{2}\,z_{||}+z_8)\cdot (\sqrt{2}\,z_{||}+z_8)\ge\Delta_8(z)=14/3>4$, and the packing constraint is satisfied for any $z_{||}$. We derive below a domain $V_\perp$ for which we can prove that $Q_{10}$ is a packing with minimal distance 2.

As a preliminary step we identify and characterize the minimal vectors of $L_{12}$ that arise in construction Q:
\begin{lema}
A non-zero $x\in L_{12}$ that satisfies the inequalities
\begin{align}
\Delta_8(x)+2\,x_{||}\cdot x_{||}&< 4\\
x_\perp\cdot x_\perp&\le 8/3
\end{align}
is a minimal vector of $L_{12}$.
\end{lema}
\begin{proof}
Combining (2) and (3) we obtain $x\cdot x< 14/3+\frac{1}{2}\Delta_8(x)$. Table 1, together with $\Delta_8(x)<4$ as implied by (2), limits the possible values of $\Delta_8$ in the bound on $x\cdot x$. Recalling that the possible norms of the integral lattice $L_{12}\cong\Lambda_{12}^\mathrm{max}$ are $4, 6, \ldots$, we see that only the case $\Delta_8(x)=10/3$ is unresolved. In this case (2) and (3) imply $x_{||}\cdot x_{||} < 1/3$ and $x_4\cdot x_4<3$, where $x_4=x_{||}+x_\perp$. Without loss of generality it is sufficient to examine just one element of the symmetry orbit of the glue, say $h=\left[\frac{1}{6}\;\frac{1}{6}\;\frac{1}{6}\;\frac{1}{6}\right]$. There are just five solutions to the inequality $x_4\cdot x_4=\| \sum_{i=1}^4\left(n_i+\frac{1}{6}\right)u_i\|^2<3$
for integers $n_i$. These are where all integers are zero, for which $x_4\cdot x_4=2/3$ and $x$ is minimal, or only one of the integers is nonzero and has the value $-1$. For the latter one obtains $x_{||}\cdot x_{||}=(4\pm\sqrt{3})/3$, both of which are inconsistent with $x_{||}\cdot x_{||} < 1/3$.
\end{proof}

The domain $V_\perp$ is constrained by the symmetry orbits of the $L_{12}$ minimal vectors, specifically their projections into $X_\perp$. Since $G_0$ acts as the reflection group of order 24 on the set of projected minimal vectors, the size of the orbits is always a multiple of 12 (since $x_\perp$ is non-zero for a minimal vector). Consulting Table 2 we see, for example, that there are $3\times 4$ minimal vectors with $x_4\cdot x_4=2$. This single orbit with respect to $G_0$ is also an orbit of $Aut(H)$ and is therefore fixed by the involution $\sigma$ that exchanges $X_{||}$ and $X_\perp$. We conclude that for minimal vectors $x$ in this conjugacy class we have $x_{||}\cdot x_{||}=x_\perp\cdot x_\perp$. The equality of projected minimal vector norms in $X_{||}$ and $X_\perp$ holds for all the classes except the class with $\Delta_8=14/3$, where this property will not be relevant, and the class with $\Delta_8=4/3$. The latter contains $4\times 9$ vectors which form three orbits with respect to $G_0$; data for these are given in Table 4.
\begin{table}
\begin{center}
\begin{tabular}[top]{rrrrccc}
\multicolumn{4}{c}{representative} &$\quad$& $x_{||}\cdot x_{||}$ & $x_\perp\cdot x_\perp$\\
\hline
$\frac{1}{3}$&$\frac{1}{3}$&$\frac{1}{3}$&$\frac{1}{3}$&
& $\frac{4}{3}$ & $\frac{4}{3}$\\
$\frac{1}{3}$&$\frac{1}{3}$&$-\frac{2}{3}$&$\frac{1}{3}$&
& $\frac{4}{3}-\frac{2}{\sqrt{3}}$& $\frac{4}{3}+\frac{2}{\sqrt{3}}$\\
$\frac{1}{3}$&$\frac{1}{3}$&$\frac{1}{3}$&$-\frac{2}{3}$&
& $\frac{4}{3}+\frac{2}{\sqrt{3}}$& $\frac{4}{3}-\frac{2}{\sqrt{3}}$\\
\end{tabular}
\caption{Data on the three orbits with respect to  the reflection group $G_0$ of the $4\times 9$ minimal vectors of $L_{12}$ with $\Delta_4=8/3$. Orbit representatives are specified in the same 4-component basis used to define the glue group.}
\end{center}
\end{table}

Only the second orbit in Table 4 requires attention, as it violates the property $\|\sqrt{2}\, x_{||}\|\ge \|x_{||}+x_\perp\|$. Avoiding such minimal vectors is one of the roles of the domain $V_\perp$. More specifically, we construct $V_\perp$ such that if $y_\perp, z_\perp\in V_\perp$, then $x_\perp=y_\perp-z_\perp$ is never one of the 12 \textit{forbidden} vectors in the orbit of $(u_1^\perp+u_2^\perp-2\, u_3^\perp+u_4^\perp)/3$.
Up to translation in $X_\perp$, there is a unique dodecagon with the property that its opposite edges are translated by vectors in this set. It is possible to include half the boundary of the dodecagon in the definition of $V_\perp$, and we make the definite choice shown in Figure 4 in what follows. $V_\perp$ has diameter $\sqrt{8/3}$, area $|V_\perp |=2$, and happens to coincide with the shadow of the Voronoi domain of $L_4=A_2\oplus A_2$, although the significance of this is not clear. 

\begin{figure}[t]
\begin{center}
\includegraphics[width=2.in]{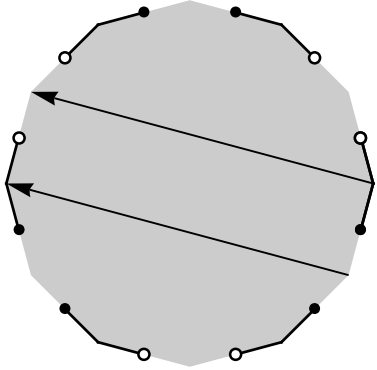}
\caption{The domain $V_\perp$, a regular dodecagon, used in construction Q. Translations associated with pairs of opposite edges (arrows) are the forbidden vectors of the construction. $V_\perp$ includes half the boundary (dark), and in particular, six dodecagon vertices.
}
\end{center}
\end{figure}

\begin{thma}
Construction Q, with $V_\perp$ as specified above, produces a sphere packing $Q_{10}$ with minimum distance 2.
\end{thma}
\begin{proof}
Consider $x, y\in S_{10}$ and let $z=x-y$. We wish to verify that $\|\sqrt{2}\,z_{||}+z_8\|\ge2$.  From the construction of $S_{10}$ we know that $z_\perp\cdot z_\perp\le 8/3$ (the diameter of $V_\perp$). The theorem is proved if we can show that $\|\sqrt{2}\,z_{||}+z_8\|<2$ leads to a contradiction. But with this statement both hypotheses of the lemma are satisfied, thereby establishing that $z$ is a minimal vector. From our exhaustive analysis of the minimal vectors we know (again omitting the irrelevant class with $\Delta_8=14/3>4$) that $\|\sqrt{2}\,z_{||}\|\ge\|z_{||}+z_\perp\|$ with one exception: the minimal vectors whose projections into $X_\perp$ coincide with one of the forbidden vectors. But $V_\perp$ was designed so this is impossible, hence $\|\sqrt{2}\,z_{||}\|\ge\|z_{||}+z_\perp\|$. Moreover, since $z$ is minimal, this shows $\|\sqrt{2}\,z_{||}+z_8\|\ge \|z_{||}+z_\perp+z_8\|=2$, a contradiction. 
\end{proof}

We know that the packing $Q_{10}$ is aperiodic because its projection into $X_{||}$ is a quasicrystal. A representative region of the projection is shown in Figure 5, where each point represents an $L_8$ lattice translated by one of 36 glue elements. To set the scale, the projections of various pairs of unit-radius spheres are also shown. Even though the sphere projections overlap, glue translations of the $L_8$ lattices guarantee that the spheres form a packing. There is a smallest separation of projected sphere centers, and spheres in this relationship amply satisfy the packing constraint because the associated conjugacy class of the gluing has $\Delta_8=14/3$. Many spheres are actually kissing; this is the case for those depicted in the figure.

\begin{figure}[t]
\begin{center}
\includegraphics[width=4.in]{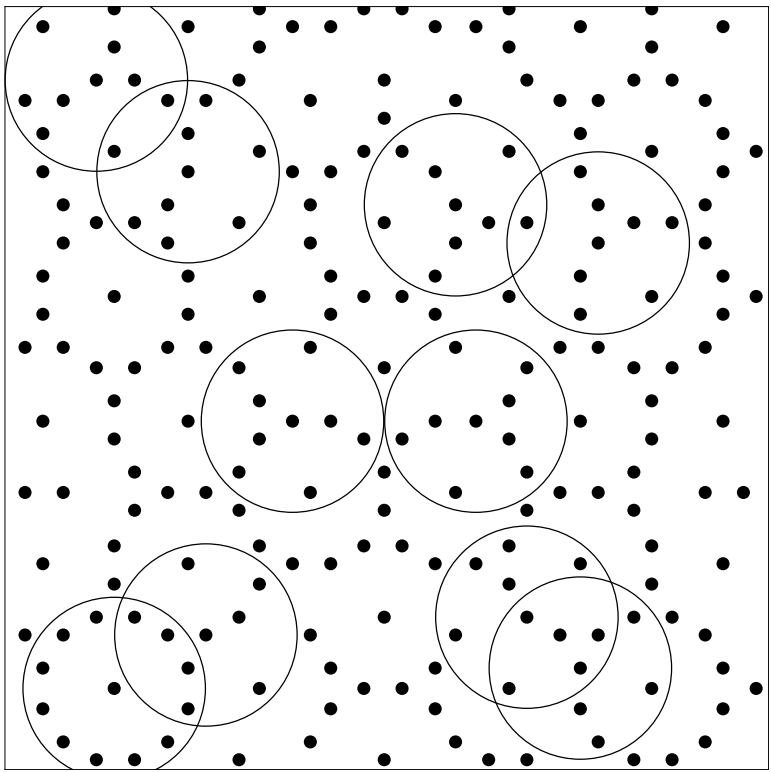}
\caption{Sphere centers (dots) of the packing $Q_{10}$ projected into the plane $X_{||}$ form an aperiodic pattern, a region of which is shown here. Pairs of kissing spheres in the 10 dimensional packing may have any of the five projections shown.
}
\end{center}
\end{figure}

The center density of $Q_{10}$ can be computed using a formula developed in the study of quasicrystals \cite{Elser1986, MoodyPatera}. Restricting to just one of the conjugacy classes of glue, and omitting the dilation by $\sqrt{2}$, the density of centers in $X_{||}$ is equal to the product of the center density of $L_4$ and the area of the domain $V_\perp$ whose shadow selects the subset of $L_4$ used in $Q_{10}$: $\delta_{||}=|V_\perp|\, \delta_4 $. Including all the conjugacy classes multiplies this by $|H|=36$ and the dilation by $\sqrt{2}$ diminishes the density by 2. To get the density in 10 dimensions this density in $X_{||}$ is multiplied by the center density of $L_8$:
\begin{equation*}
\delta_{10}=\frac{1}{2}|V_\perp|\,|H| \, \delta_4\, \delta_8=\frac{1}{2}\, |V_\perp|\, \delta_{12}=\frac{1}{32}.
\end{equation*}

\begin{figure}[t]
\begin{center}
\includegraphics[width=4.in]{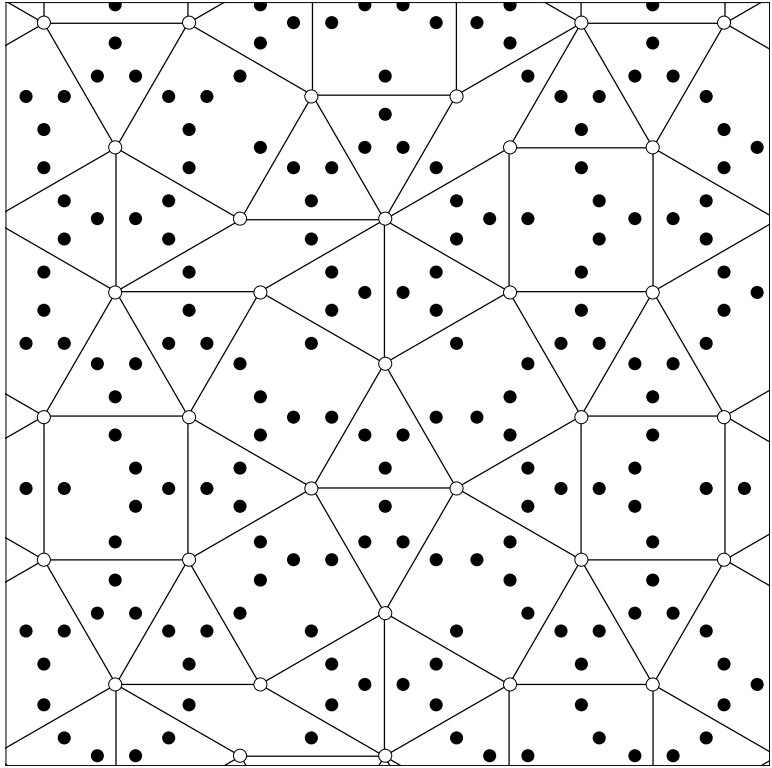}
\caption{The projected sphere centers in $X_{||}$ form a quasicrystal tiling comprising triangles, squares, and $30^\circ$ rhombi.
}
\end{center}
\end{figure}

Because $Q_{10}$ is a non-lattice packing, the sphere centers fall into distinct classes. This is evident in the quasicrystal pattern of projected centers in $X_{||}$, where there is a special class of centers that is never in the shortest-separation relationship with other centers. When centers in this special class are connected by edges, the result, shown in Figure 6, is a tiling of triangles, squares and $30^\circ$ rhombi. Alternate tilings formed by these three tiles, utilizing the same gluing scheme for the $L_8$ lattices, also correspond to valid packings (details omitted). The densest of these is obtained by the tiling that only uses the triangle, the densest tile of the three. The improvement of the density, to the value $7/(96+64\sqrt{3})\approx 0.0338$, comes at the expense of symmetry. Whereas the quasicrystal pattern has point symmetry group $G_0$ of order 24, the point group of the triangular tiling is only of order 12. The densest known packing in 10 dimensions has center density $5/128\approx 0.03906$ \cite{SPLAG}.

We conclude by returning to the results of the numerical experiment that prompted this investigation. The arrangement of 378 kissing spheres discovered by the search algorithm \cite{D&C} coincides with the arrangement obtained from the packing $Q_{10}$ when the dodecagonal domain $V_\perp$ is given a singular centering (translation) in $X_\perp$. At a singular centering one $L_8$ lattice projects to the exact center of $V_\perp$, thereby making the six vertices of $V_\perp$ available to the packing (see Fig. 4). Relative to the dodecagon center, the six vertices have glue in the orbit of $\left[\frac{1}{3}\;\frac{1}{3}\;0\;0\right]$ with $\Delta_8=8/3$ and $\Delta_4=4/3$. From Table 1 we see that $\tau_8=9$ spheres, in each of the $L_8$ lattices that project to these vertices of $V_\perp$, make contact with the central sphere. There is also a set of 12 $L_8$ lattices that project to a regular dodecagon within $V_\perp$. These are in the orbit of $\left[\frac{5}{6}\;\frac{1}{3}\;\frac{5}{6}\;\frac{1}{3}\right]$ with $\Delta_8=10/3$, $\Delta_4=2/3$ and $\tau_8=24$. The $X_{||}$ counterpart of this relationship are the many examples of complete dodecagons encircling tile vertices in Fig. 6.  Finally, the $L_8$ lattice that projects to the dodecagon center makes 36 contacts with the central sphere (the kissing number of $L_8$). The net kissing number is thus $6\times 9+12\times 24+36=378$.

\end{document}